\documentclass[12pt,a4paper, reqno]{amsart}
\usepackage{amsfonts,amsmath,amsthm,amssymb, amscd}
\usepackage[top=2cm, bottom=5cm, left=1cm, right=1cm]{geometry}

\newtheorem{theorem}{Theorem}
\newtheorem{lemma}{Lemma}
\newtheorem{corollary}{Corollary}

\newtheorem{definition}{Definition}

\newtheorem{remark}{Remark}

\title{Determinants containing powers of polynomial sequences}

\thanks{
Supported by the UAEU grants: StartUp Grant 2016: G00002235}

\author{Ho-Hon Leung}
\address{Department of Mathematical Sciences, United Arab Emirates University, Al Ain, 15551, United Arab Emirates}
\email{hohon.leung@uaeu.ac.ae}

\begin{document}
\maketitle

\begin{abstract}
We derive identities for the determinants of matrices whose entries are (rising) powers of (products of) polynomials that satisfy a recurrence relation. In particular, these results cover the cases for Fibonacci polynomials, Lucas polynomials and certain orthogonal polynomials. These identities naturally generalize the determinant identities obtained by Alfred, Carlitz, Prodinger, Tangboonduangjit and Thanatipanonda. 
\end{abstract}

\section{Introduction} \label{section1}

Let $(F_n)_{n\geq 0}$ be the Fibonacci sequence. Let $m\geq 1$ and $n$ be any nonnegative integer. Let $[F^m_{n+i+j}]_{0\leq i,j\leq m}$ be the $(m+1)\times (m+1)$ matrix with entries $F_{n+i+j}$, $0\leq i,j\leq m$. In 1966, Carlitz \cite{Carlitz} derived the following determinant identity for the matrix $[F^m_{n+i+j}]_{0\leq i,j\leq m}$:
\begin{align}
\label{equation1}   \det([F^m_{n+i+j}]_{0\leq i,j\leq m}) &= (-1)^{(n+1)\binom{m+1}{2}}(F_1^m F_2^{m-1} \cdots F_m)^2 \prod_{i=0}^m \binom{m}{i}.
\end{align}This identity is related to the problems posted by Alfred \cite[p. 48]{Alfred1} in 1963 and Parker \cite[p. 303]{Parker} in 1964 respectively. Let $s$ and $k$ be any integers. Tangboonduangjit and Thanatipanonda \cite{Aek1} generalized the determinant identity (\ref{equation1}) as follows: 
\begin{align}
\label{equation2}   \det([F^m_{s+k(n+i+j)}]_{0\leq i,j\leq m})   &= (-1)^{(s+kn+1)\binom{m+1}{2}} (F^m_k F^{m-1}_{2k} \cdots F_{mk})^2 \prod_{i=0}^m \binom{m}{i}.
\end{align}Let $F_n^{\langle m \rangle}$ be the {\it rising powers} of the Fibonacci numbers defined by
\begin{align*}
 F_n^{\langle m\rangle} := F_n F_{n+1} \cdots F_{n+m-1}.
\end{align*}Prodinger \cite{Prodinger} obtained the following determinant identity for the matrix $[F_{n+i+j}^{\langle m\rangle}]_{0\leq i,j\leq m}$:
\begin{align}
\label{equation3}  \det([F_{n+i+j}^{\langle m\rangle}]_{0\leq i,j\leq m})&= (-1)^{n\binom{m+1}{2}+\binom{m+2}{3}} (F_1 F_2 \cdots F_m)^{m+1}.
\end{align}Tangboonduangjit and Thanatipanonda \cite{Aek2} generalized the determinant identity (\ref{equation3}) as follows:
\begin{align}
\label{equation4}  \det([F_{n+i+j}^{\langle m\rangle}]_{0\leq i,j\leq d-1})&= (-1)^{n\binom{d}{2}+\binom{d+1}{3}} \prod_{i=1}^{d-1} (F_i F_{m+1-i})^{d-i}  \cdot \prod_{i=d-1}^{2(d-1)} F_{n+i}^{\langle m+1-d \rangle}.
\end{align}where $d\geq 2$. It is worthwhile to note that Tangboonduangjit and Thanatipanonda \cite{Aek1, Aek2} derived the determinant identities more generally, for matrices whose entries include (rising) powers of terms that satisfy a second-order linear recurrence relation with constant coefficients. By using analogous techniques in determinant calculus, we derive determinant identities for matrices whose entries are (rising) powers of polynomials that satisfy certain recurrence relations. As corollaries, we provide determinant identities for matrices whose entries are (rising) powers of Fibonacci polynomials, Lucas polynomials and certain orthogonal polynomials. As an application, we obtain new identities in the case of Fibonacci numbers. For example, for $n\geq 1$, by Corollary \ref{corollary2}, we get 
\begin{align}
\label{new}   \det\Big(\Big[ \frac{1}{F_{n+i+j}} \Big]_{0\leq i, j\leq m}\Big) &= \frac{(-1)^{n \binom{m+1}{2}}\prod_{i=0}^m F_{i+1}^{2(m-i)}  }{\prod_{0\leq i,j\leq m} F_{n+i+j}}.
\end{align}

\section{Main Results} \label{section2}

\begin{definition}
Let $p,q, r, a, b$ and $c$ be any real numbers. Let $\mathbb{Z}_{\geq 0}=\{0,1,2,\ldots\}$. The sequence of polynomials in variable $x$, \[\mathcal{P}(x)(p,q,r; a,b,c):=(P_n)_{n\in\mathbb{Z}},\]is defined by
\begin{align*}
P_0 :=p, \quad P_1:=qx+r, \quad P_{n+2}:=(a x+b)P_{n+1} + c P_n, \text{ for }n\in\mathbb{Z}_{\geq 0}.
\end{align*}For $n<0$, $P_n$ is defined by 
\begin{align*}
 P_{n} := -\frac{ax+b}{c} P_{n+1} +\frac{1}{c} P_{n+2}.
\end{align*}The discriminant $\Delta_\mathcal{P}$ is defined by \[\Delta_{\mathcal{P}}:=(q^2 -apq)x^2 +(2qr -apr-bpq)x+(r^2 -bpr -cp^2).\]
\end{definition}

\begin{theorem}\label{theorem1}
Let $\mathcal{P}(x)(p_1,q_1,r_1;a,b,c)=(P_n)_{n\in\mathbb{Z}}$, $\mathcal{Q}(x)(p_2,q_2,r_2;a,b,c)=(Q_n)_{n\in\mathbb{Z}}$ and  $\mathcal{U}(x)(0,0,1;a,b,c)=(U_n)_{n\in\mathbb{Z}}$ be the sequences of polynomials defined by real numbers $p_1, p_2, q_1, q_2, r_1, r_2, a, b, c$ where $c\neq 0$. Then 
\begin{align*}
 P_{s+i} Q_{s+j} - P_{s} Q_{s+i+j} &= (-c)^s (P_1 Q_j - P_0 Q_{j+1}) U_i
\end{align*}for all integers $s, i , j$.
\end{theorem}

\begin{proof}
We prove it by induction on $i$. It is trivial for $i=0$. If $i=1$, we have 
\begin{align}
\label{equation5}  \begin{pmatrix}     
 P_{s+1}  &   Q_{s+j+1} \\
 P_{s}  &   Q_{s+j}
\end{pmatrix}&=  \begin{pmatrix}  
ax+b &  c\\
 1 & 0
\end{pmatrix} \begin{pmatrix}
 P_{s} & Q_{s+j} \\
 P_{s-1} & Q_{s+j-1} 
\end{pmatrix} =\begin{pmatrix}
ax+b & c \\
1 & 0
\end{pmatrix}^s  \begin{pmatrix}
P_1  &  Q_{j+1} \\
P_0 & Q_j 
\end{pmatrix}\text{ for }s\geq 0,   \\
\label{equation6}  \begin{pmatrix}     
 P_{s+1}  &   Q_{s+j+1} \\
 P_{s}  &   Q_{s+j}
\end{pmatrix}&=  \begin{pmatrix}  
ax+b & c \\
1 & 0 
\end{pmatrix}^{-1} \begin{pmatrix}
P_{s+2} & Q_{s+j+2} \\
P_{s+1} & Q_{s+j+1}
\end{pmatrix}= \begin{pmatrix}
ax+b & c \\
1 & 0
\end{pmatrix}^s \begin{pmatrix}
P_1  & Q_{j+1} \\
P_0 & Q_{j} 
\end{pmatrix}\text{ for }s<0.
\end{align}We take the determinants on both sides of (\ref{equation5}) and (\ref{equation6}) to get 
\begin{align}
\label{equation7}  P_{s+1} Q_{s+j} - P_{s} Q_{s+j+1} &= (-c)^s (P_1 Q_j - P_0 Q_{j+1}) =(-c)^s (P_1 Q_j - P_0 Q_{j+1}) U_1. 
\end{align}For $i>1$, we assume that the identity is true for $i-1$ and $i-2$. We have
\begin{align*}
P_{s+i} Q_{s+j} - P_{s} Q_{s+i+j} &= \begin{vmatrix}
 P_{s+i} & Q_{s+i+j} \\
 P_s & Q_{s+j}
\end{vmatrix} \\
 &= \begin{vmatrix}
(ax+b) P_{s+i-1} +c P_{s+i-2} & (ax+b) Q_{s+i+j-1} +c Q_{s+i+j-2} \\
  P_s & Q_{s+j}
\end{vmatrix} \\
 &= (ax+b) \begin{vmatrix}
 P_{s+i-1} & Q_{s+j+(i-1)} \\
 P_s & Q_{s+j} 
\end{vmatrix} + c \begin{vmatrix}
P_{s+i-2}  & Q_{s+j+(i-2)} \\
P_s  & Q_{s+j} 
\end{vmatrix} \\
&= (ax+b) (-c)^{s}(P_1 Q_j - P_0 Q_{j+1}) U_{i-1} +c (-c)^s (P_1 Q_j - P_0 Q_{j+1}) U_{i-2} \\
&= (-c)^s (P_1 Q_j - P_0 Q_{j+1} )((ax+b)U_{i-1} +c U_{i-2}) \\
&= (-c)^s (P_1 Q_j - P_0 Q_{j+1}) U_i.
\end{align*}For $i<0$, we assume that the identity is true for $i+1$ and $i+2$. We have
\begin{align*}
P_{s+i} Q_{s+j} - P_{s} Q_{s+i+j}  &= \begin{vmatrix}
 P_{s+i} & Q_{s+i+j} \\
 P_{s} & Q_{s+j} 
\end{vmatrix} \\
&= \begin{vmatrix}
 -\frac{ax+b}{c} P_{s+i+1}+\frac{1}{c} P_{s+i+2} & -\frac{ax+b}{c} Q_{s+i+j+1} +\frac{1}{c} Q_{s+i+j+2} \\
 P_s & Q_{s+j} 
\end{vmatrix}\\
&= -\frac{ax+b}{c}\begin{vmatrix}
P_{s+i+1} & Q_{s+i+j+1} \\
 P_s & Q_{s+j}
\end{vmatrix} +\frac{1}{c} \begin{vmatrix}
P_{s+i+2} & Q_{s+i+j+2} \\
P_s & Q_{s+j} 
\end{vmatrix} \\
&= -\frac{ax+b}{c} (-c)^s (P_1 Q_j - P_0 Q_{j+1} ) U_{i+1} +\frac{1}{c} (-c)^s (P_1 Q_j - P_0 Q_{j+1}) U_{i+2} \\
&=(-c)^s (P_1 Q_j - P_0 Q_{j+1}) \Big(-\frac{ax+b}{c}U_{i+1} +\frac{1}{c} U_{i+2}\Big) \\
&=(-c)^s (P_1 Q_j - P_0 Q_{j+1} ) U_i. 
\end{align*}
\end{proof}

\begin{remark}
We recover the {\it generalized Catalan Identity} by Melham and Shannon \cite{Shannon} (see also Tangboonduangjit and Thanatipanonda \cite[Proposition 1]{Aek1}) by substituting $x=1$ in Theorem \ref{theorem1}.
\end{remark}

\begin{corollary} \label{corollary1}
Let $\mathcal{P}(x)(p,q,r;a,b,c)=(P_n)_{n\in\mathbb{Z}}$ and $\mathcal{U}(x)(0,0,1;a,b,c)=(U_n)_{n\in\mathbb{Z}}$ be the sequences of polynomials defined by real numbers $p, q, r, a, b, c$ where $c\neq 0$. Then
\begin{align}
\label{equation7.1} P_{j} P_1 - P_0 P_{j+1} &= \Delta_{\mathcal{P}} \cdot U_j, \\
\label{equation7.2} P_{s+i} P_{s+j} - P_{s} P_{s+i+j} &= (-c)^s \Delta_\mathcal{P} \cdot U_i U_j
\end{align}where $\Delta_\mathcal{P}$ is the discriminant of $\mathcal{P}(x)(p,q,r;a,b,c)$.
\end{corollary}

\begin{proof}
By setting $s=0, j=1, p_1=p_2=p, q_1=q_2=q, r_1=r_2=r$ in Theorem \ref{theorem1}, we get 
\begin{align}
\label{equation8} P_i P_1 -P_0 P_{i+1} &= (P_1 P_1 - P_0 P_2)U_i.
\end{align}We note that $P_0=p$, $P_1=qx+r$ and $P_2=(ax+b)P_1+c P_0=(ax+b)(qx+r)+cp$. Hence, we obtain (\ref{equation7.1}) by simplifying (\ref{equation8}). On the other hand, by setting $p_1=p_2=p, q_1=q_2=q, r_1=r_2=r$ in Theorem \ref{theorem1}, we get 
\begin{align*}
 P_{s+i} P_{s+j} - P_{s} P_{s+i+j} &= (-c)^s (P_1 P_j - P_0 P_{j+1})U_i=(-c)^s \Delta_\mathcal{P} \cdot U_i U_j
\end{align*}in which the last equality is based on (\ref{equation7.1}).
\end{proof}

\begin{lemma} \label{lemma1}
Let $m\geq 1$. Let $B_i, D_i$ be polynomials in variable $x$, $A_i, C_i$ be rational functions in variable $x$, for $0\leq i\leq m$. Let $[(A_j B_i+ C_j D_i)^m]_{0\leq i,j\leq m}$ be the $(m+1)\times (m+1)$ matrix with entries $(A_j B_i+ C_j D_i)^m, 0\leq i,j\leq m$. Then we have the following determinant identity:
\begin{align}
\label{equation8.9}  \det([(A_j B_i+ C_j D_i)^m]_{0\leq i,j\leq m})&=\prod_{0\leq i<j\leq m}\Big((B_i D_j-B_j D_i)(A_i C_j - A_j C_i)\Big) \cdot \prod_{i=0}^m \binom{m}{i}.
\end{align}
\end{lemma}

\begin{proof}
We invoke the following result by Krattenthaler \cite[Lemma 10]{Krattenthaler} (see also Tangboonduangjit and Thanatipanonda \cite[Lemma 3]{Aek1}:
\begin{align}
\label{equation9} \det([(c_j d_i+1)^m]_{0\leq i,j\leq m}) &= \prod_{0\leq i<j\leq m}\Big( (d_i -d_j)(c_i-c_j)\Big)\cdot\prod_{i=0}^m\binom{m}{i}
\end{align}where $c_j, d_i$ are real numbers for $0\leq i,j\leq m$. First, we prove the lemma for polynomials $A_i, B_i, C_i, D_i$ for all $0\leq i\leq m$. For the values of $x$ such that $C_j\neq 0$ and $D_i\neq 0$ for $0\leq i,j\leq m$, let \[c_j=\frac{A_j}{C_j}, \quad d_i=\frac{B_i}{D_i},\text{ for }0\leq i,j\leq m.\]We note that
\begin{align}
\nonumber  \det([(c_j d_i+1)^m]_{0\leq i,j\leq m}) &=\det\Big(\Big[\Big(\frac{A_jB_i}{C_j D_i}+1\Big)^m\Big]_{0\leq i,j\leq m}\Big)=\det\Big(\Big[\Big(\frac{A_j B_i+C_j D_i}{C_j D_i}\Big)^m \Big]_{0\leq i,j\leq m}\Big)\\
\label{equation9.1}  &= \Big(\prod_{0\leq i\leq m}\frac{1}{(C_i D_i)^m} \Big)\cdot \det([(A_j B_i+ C_j D_i)^m]_{0\leq i,j\leq m}).
\end{align}Also, we have 
\begin{align}
\nonumber \prod_{0\leq i<j\leq m} \Big( (d_i -d_j)(c_i-c_j) \Big)&=\prod_{0\leq i<j\leq m} \Big(\Big( \frac{B_i D_j-B_j D_i}{D_i D_j}\Big)\cdot\Big(\frac{A_i C_j - A_j C_i}{C_i C_j}  \Big)  \Big)  \\
\label{equation9.2} &=\prod_{0\leq i\leq m}\Big(\frac{1}{(C_i D_i)^m} \Big)     \cdot\prod_{0\leq i<j\leq m}\Big((B_i D_j-B_j D_i)(A_i C_j - A_j C_i)\Big).
\end{align}By (\ref{equation9}), (\ref{equation9.1}), (\ref{equation9.2}), we get (\ref{equation8.9}) as desired.

Based on the facts that there are only a finite number of roots for $C_j, D_i$ where $0\leq i,j\leq m$ and the determinant of a matrix  with polynomial entries is a continuous function in $x$, the equality (\ref{equation8.9}) still holds true for the values of $x$ such that $C_j=0$ or $D_i=0$ for some $i$ or $j$.

Next, we assume that $A_i$ and $C_i$ are rational functions for all $0\leq i\leq m$. We write $A_i$ and $C_i$ as follows: 
\[ A_i=\frac{E_i}{F_i}, \quad C_i =\frac{G_i}{H_i} \text{ for }0\leq i\leq m,\]where $E_i, F_i, G_i, H_i$ are all polynomials for $0\leq i\leq m$. For the values of $x$ such that $F_i\neq 0$ and $H_i\neq 0$ for all $0\leq i\leq m$, we get 
\begin{equation}
\begin{aligned}
\nonumber\det([(A_j B_i + C_j D_i)^m]_{0\leq i,j\leq m}) ={}&\det\Big(\Big[ \Big(\frac{H_jE_j B_i+D_i F_j G_j}{F_j H_j} \Big)^m   \Big]_{0\leq i,j\leq m}\Big) \\
 \nonumber={}& \Big(\prod_{0\leq i\leq m}\frac{1}{(F_i H_i)^m} \Big) \cdot \det([((H_j E_j)B_i +(G_j F_j) D_i)^m]_{0\leq i,j\leq m}   )\\
 \nonumber={}&  \Big(\prod_{0\leq i\leq m}\frac{1}{(F_i H_i)^m} \Big) \Big(\prod_{i=0}^m \binom{m}{i}\Big) \Big( \prod_{0\leq i<j\leq m} (B_i D_j-B_j D_i)\cdot\\  &(H_i E_i G_j F_j - H_j E_j G_i F_i) \Big)\\
 \nonumber ={}& \Big(\prod_{i=0}^m \binom{m}{i}\Big) \Big(\prod_{0\leq i<j\leq m} (B_i D_j-B_j D_i) \big( \frac{E_i}{F_i}\cdot\frac{G_j}{H_j} -\frac{E_j}{F_j}\cdot \frac{G_i}{H_i} \big) \Big) \\
\nonumber ={}& \Big(\prod_{i=0}^m \binom{m}{i}\Big) \Big(\prod_{0\leq i<j\leq m} (B_i D_j-B_j D_i) \big( A_i C_j-A_j C_i \big) \Big). 
\end{aligned}
\end{equation}For the values of $x$ such that $F_i=0$ or $H_i=0$ for some $i$, the equality still holds true as the determinant of a matrix with polynomial entries is a continuous function. 
\end{proof}

\begin{theorem}\label{theorem2}
Let $s, k, n$ be any integers, $m\geq 1$. Let $\mathcal{P}(x)(p,q,r;a,b,c)=(P_n)_{n\in\mathbb{Z}}$ and $\mathcal{U}(x)(0,0,1;a,b,c)=(U_n)_{n\in\mathbb{Z}}$ be the sequences of polynomials defined by real numbers $p, q, r, a, b, c$ where $c\neq 0$.  The determinant of the matrix $[P_{s+k(n+i+j)}^m]_{0\leq i,j\leq m}$ is given by \[\det([P_{s+k(n+i+j)}^m]_{0\leq i,j\leq m})=(-1)^{(s+kn+1)\binom{m+1}{2}} \Delta_\mathcal{P}^{\binom{m+1}{2}}\cdot c^{(s+kn)\binom{m+1}{2}+2k\binom{m+1}{3}}\cdot\prod_{i=0}^m \Big(\binom{m}{i} U_{k(i+1)}^{2(m-i)} \Big)\]where $\Delta_\mathcal{P}$ is the discriminant of $\mathcal{P}(x)(p,q,r;a,b,c)$.
\end{theorem}

\begin{proof}
By substituting $s=-km, i=kj', j=s'+k(n+m+i')$ into (\ref{equation7.2}) in Corollary \ref{corollary1} and then replacing $s',i',j'$ by $s,i,j$ respectively, we get
\begin{align}
\nonumber P_{k(j-m)} P_{s+k(n+i)} - P_{-km} P_{s+k(n+i+j)}&=(-c)^{-km} \Delta_{\mathcal{P}}\cdot U_{kj} U_{s+k(n+m+i)},\\
\label{equation10}  P_{s+k(n+i+j)}&= \frac{P_{k(j-m)}}{P_{-km}}\cdot P_{s+k(n+i)}+\frac{-(-c)^{-km}\Delta_\mathcal{P}\cdot U_{kj}}{P_{-km}} \cdot U_{s+k(n+m+i)}. 
\end{align}By substituting $s=s'+k(n+m+i'), i=k(j'-i'), j=-km$ into Theorem \ref{theorem1} and then replacing $s',i',j'$ by $s,i,j$ respectively, we get 
\begin{align}
\nonumber P_{s+k(n+i)} U_{s+k(n+m+j)} -P_{s+k(n+j)} U_{s+k(n+m+i)} &=(-c)^{s+k(n+m+i)}(U_1 P_{-km} -U_0 P_{-km+1})U_{k(j-i)}\\
\label{equation11} &= (-c)^{s+k(n+m+i)} P_{-km} U_{k(j-i)}.
\end{align}By substituting $s=ki', i=k(j'-i'), j=-km$ into Theorem \ref{theorem1} and then replacing $i',j'$ by $i,j$ respectively, we get \begin{align}
\label{equation12} P_{k(i-m)} U_{kj} -P_{k(j-m)} U_{ki} &= (-c)^{ki} P_{-km} U_{k(j-i)}. 
\end{align}By (\ref{equation10}), we get
\begin{align}
\label{equation13} \det([P_{s+k(n+i+j)}^m]_{0\leq i,j\leq m})&=\det\Big(\Big[\Big(\frac{P_{k(j-m)}}{P_{-km}}\cdot P_{s+k(n+i)}+\frac{-(-c)^{-km}\Delta_\mathcal{P} \cdot U_{kj}}{P_{-km}} \cdot U_{s+k(n+m+i)}\Big)^m\Big]_{0\leq i,j\leq m}\Big).
\end{align}By (\ref{equation8.9}), the term in (\ref{equation13}) becomes
\begin{align}
\label{equation14} \prod_{i=0}^m \binom{m}{i}\cdot \prod_{0\leq i<j\leq m} \Big((P_{s+k(n+i)} U_{s+k(n+m+j)}-P_{s+k(n+j)} U_{s+k(n+m+i)})\Big( \frac{-(-c)^{-km}\Delta_{\mathcal{P}}}{P_{-km}^2} \big(P_{k(i-m)} U_{kj}-P_{k(j-m)} U_{ki}    \big)\Big)   \Big).
\end{align}By (\ref{equation11}), (\ref{equation12}), the term in (\ref{equation14}) becomes 
\begin{align}
\prod_{i=0}^m \binom{m}{i}\cdot   \prod_{0\leq i<j\leq m}\Big( (-1)^{s+kn+1} c^{s+k(n+2i)} \Delta_\mathcal{P} \cdot U^2_{k(j-i)}\Big).
\end{align}As a consequence, we get the desired result by standard counting arguments. 
\end{proof}

\begin{remark}
We recover Theorem 5 in the work of Tangboonduangjit and Thanatipanonda \cite{Aek1} by substituting $x=1$ in Theorem \ref{theorem2}.
\end{remark}

Next, we look at other determinant identities. 

\begin{lemma} \label{lemma2}
Let $m\geq 1$. Let $B_i, D_i$ be polynomials in variable $x$ for $0\leq i\leq m$. Let $A_j, C_j$ be rational functions in variable $x$ for $i\in \mathbb{Z}$. Let $(d_i)_{1\leq i\leq r}$ and $(e_i)_{1\leq i\leq r}$ be sequences of integers. Then \[\det\Big(\Big[ \prod_{f=j+1}^m \big( A_{d_f} B_i+C_{d_f} D_i \big) \cdot \prod_{g=1}^j \big( A_{e_g} B_i +C_{e_g} D_i  \big)   \Big]_{0\leq i,j\leq m}   \Big)= \prod_{0\leq i<j\leq m} \big(B_i D_j - B_j D_i   \big)\cdot \prod_{1\leq i\leq j\leq m}\big(C_{e_i} A_{d_j} -A_{e_i} C_{d_j}    \big).   \]
\end{lemma}

\begin{proof}
By the factorization method of Krattenthaler \cite[Section 4]{Krattenthaler}, it is plain to get the following identity: \begin{align} \label{equation21} \det \Big(  \Big[ \prod_{f=j+1}^m \big( X_i +F_f \big) \cdot \prod_{g=1}^j \big( X_i+G_g \big)   \Big]_{0\leq i,j\leq m} \Big)&= \prod_{0\leq i<j\leq m} \big(X_j - X_i \big) \cdot \prod_{1\leq i\leq j\leq m}  \big( F_j - G_i  \big)\end{align}where $X_i$ for $0\leq i\leq m$, $D_j, E_j$ for $1\leq j\leq m$ are some indeterminates. For the values of $x$ such that $D_i\neq 0$ and $A_j\neq 0$ for $0\leq i\leq m$ and $j\in\mathbb{Z}$, let \[ X_i=\frac{B_i}{D_i}, \quad F_j=\frac{C_{d_j}}{A_{d_j}}, \quad  G_j=\frac{C_{e_j}}{A_{e_j}} \]for $0\leq i\leq m$ and $1\leq j\leq m$. By similar reasoning as in the proof of Lemma \ref{lemma1}, we get the desired result by clearing the denominators on both sides of (\ref{equation21}). For the values of $x$ which are the roots of $D_i$ or $A_j$ for some $i$ or $j$, the equality still holds true based on the fact that the determinant of a matrix with polynomial entries is a continuous function. 
\end{proof}

\begin{theorem} \label{theorem3}
Let $s, k, n$ be any integers, $m\geq 1$. Let $\mathcal{P}(x)(p,q,r;a,b,c)=(P_n)_{n\in\mathbb{Z}}$ and $\mathcal{U}(x)(0,0,1;a,b,c)=(U_n)_{n\in\mathbb{Z}}$ be the sequences of polynomials defined by real numbers $p, q, r, a, b, c$ where $c\neq 0$. Let $(d_i)_{1\leq i\leq m}$ and $(e_i)_{1\leq i\leq m}$ be sequences of integers. Then
\begin{equation}
\begin{aligned}
 \det \Big( \Big[\prod_{f=j+1}^m P_{s+k(n+i+d_f)} \prod_{g=1}^j P_{s+k(n+i+e_g)}\Big]_{0\leq i,j\leq m}   \Big) ={} & (-\Delta_\mathcal{P})^{\binom{m+1}{2}} (-c)^{(s+kn)\binom{m+1}{2} +k\binom{m+1}{3}} \prod_{l=1}^m U_{kl}^{m+1-l}\cdot\\ \nonumber  &\prod_{1\leq i\leq j\leq m} (-c)^{kd_j} U_{k(e_i-d_j)}
\end{aligned}
\end{equation}where $\Delta_\mathcal{P}$ is the discriminant of $\mathcal{P}(x)(p,q,r;a,b,c)$.
\end{theorem}

\begin{proof}
By (\ref{equation10}), Lemma \ref{lemma2} and Corollary \ref{corollary1}, the theorem can be proved in the same way as in the proof the Theorem \ref{theorem2}.
\end{proof}

\begin{lemma} \label{lemma2.5}
Let $m\geq 1$. Let $A_i, B_i$ are polynomials in variable $x$ for $0\leq i\leq m$. Let $C_i, D_i$ be rational functions in variable $x$ for $0\leq i \leq m$. Then, 
\[\det \Big(\Big[ \frac{1}{A_i D_j + B_i C_j} \Big]_{0\leq i,j\leq m}  \Big)=\frac{\prod_{0\leq i<j\leq m}(A_i B_j - A_j B_i )  (C_i D_j - D_i C_j)  }{ \prod_{0\leq i,j\leq m} (A_i D_j + B_i C_j)}\]provided that the denominators on both sides of the identity are nonzero.
\end{lemma}

\begin{proof}
First, we invoke a result of Krattenthaler \cite[Theorem 12]{Krattenthaler}. That is,
\begin{align}
\label{equation101}  \det\Big(\Big[ \frac{1}{x_i +y_j} \Big]_{0\leq i,j\leq m}  \Big)  &= \frac{\prod_{0\leq i<j\leq m} (x_i-x_j)(y_i- y_j)}{\prod_{0\leq i,j\leq m} (x_i+y_j)}
\end{align}where $x_i$ and $y_i$ are indeterminates for $0\leq i,j\leq m$. We first assume that $A_i, B_i, C_i, D_i$ are all polynomials for all $0\leq i\leq m$. For the values of $x$ such that $B_i, D_i$ are nonzero for all $0\leq i\leq m$, let \[x_i=\frac{A_i}{B_i}, \quad y_i=\frac{C_i}{D_i} \text{ for }0\leq i\leq m.\]By similar reasoning as shown in the proof of Lemma \ref{lemma1}, we get the desired result by some algebraic simplification for the cases where $A_i, B_i, C_i, D_i$ are polynomials for all $0\leq i\leq m$. 

We extend the proof to the cases where $C_i$ and $D_i$ are rational functions by the same arguments as in the proof of Lemma \ref{lemma1}, based on the fact that the determinant of a matrix with rational functions as entries is a continuous function provided that the denominators on both sides of the identity are nonzero. 
\end{proof}

\begin{theorem} \label{theorem3.5}
Let $s, k, n$ be any integers, $m\geq 1$. Let $\mathcal{P}(x)(p,q,r;a,b,c)=(P_n)_{n\in\mathbb{Z}}$ and $\mathcal{U}(x)(0,0,1;a,b,c)=(U_n)_{n\in\mathbb{Z}}$ be the sequences of polynomials defined by real numbers $p, q, r, a, b, c$ where $c\neq 0$.  The determinant of the matrix $[1/ P_{s+k(n+i+j)}]_{0\leq i,j\leq m}$ is given by \[\det\Big(\Big[\frac{1}{P_{s+k(n+i+j)}}\Big]_{0\leq i,j\leq m}\Big)=\frac{(-1)^{(s+kn)\binom{m+1}{2}} c^{(s+kn)\binom{m+1}{2}+2k\binom{m+1}{3}}\Delta_\mathcal{P}^{\binom{m+1}{2}} \prod_{i=0}^m U_{k(i+1)}^{2(m-i)}   }{\prod_{0\leq i,j\leq m} P_{s+k(n+i+j)}}\]where $\Delta_\mathcal{P}$ is the discriminant of $\mathcal{P}(x)(p,q,r;a,b,c)$, provided that the denominators on both sides of the identity are nonzero.
\end{theorem}

\begin{proof}
The proof is essentially the same as the proof of Theorem \ref{theorem2} by applying (\ref{equation10}), (\ref{equation11}), (\ref{equation12}) to Lemma \ref{lemma2.5} and some standard counting arguments.
\end{proof}

Let $A$ be a $m\times m$ matrix. Let $A_k(i,j)$ be the determinant of the $k\times k$ submatrix of $A$ whose first entry is at the position of the $i^{th}$-row and the $j^{th}$-column of $A$. 

\begin{lemma} \label{lemma3}
Let $A$ be a $m\times m$ matrix whose entries are rational functions in variable $x$. Then
\[A_{m} (1,1) A_{m-2} (2,2)= A_{m-1} (1,1) A_{m-1} (2,2) - A_{m-1} (2,1) A_{m-1} (1,2).\]
\end{lemma}

\begin{proof}
We invoke the Desnanot-Jacobi identity \cite{Zeilberger} for a matrix $A$ of size $m \times m$ with indeterminates as entries. \[A_m (1,1) A_{m-2} (2,2) = A_{m-1} (1,1) A_{m-1} (2,2) -A_{m-1} (2,1) A_{m-1} (1,2).\]To extend this result to the case where the matrix $A$ has rational functions as entries, we simply use the same strategy as in the proof of Lemma \ref{lemma1}.
\end{proof}

Let $m\geq 1$. The {\it rising powers} of a sequence of polynomials $\mathcal{P}(x)(p,q,r;a,b,c)=(P_n)_{n\in\mathbb{Z}}$ is denoted by $P_n^{\langle m\rangle}$, which is defined by \[P_n^{\langle m\rangle}:= P_n P_{n+1} \cdots P_{n+m-1}.\]

\begin{theorem} \label{theorem4}
Let $n$ be any integer. Let $m\geq 1$ and $d\geq 1$. Let $\mathcal{P}(x)(p,q,r;a,b,c)=(P_n)_{n\in\mathbb{Z}}$ and $\mathcal{U}(x)(0,0,1;a,b,c)=(U_n)_{n\in\mathbb{Z}}$ be the sequences of polynomials defined by real numbers $p, q, r, a, b, c$ where $c\neq 0$. Then
\[ \det ([ P_{n+i+j}^{\langle m\rangle}]_{0\leq i,j\leq d-1} )=(-1)^{n\binom{d}{2}+\binom{d+1}{3}} c^{(n+d-2)\binom{d}{2}} \Delta_\mathcal{P}^{\binom{d}{2}} \cdot \prod_{i=1}^{d-1} \big( U_{i} U_{r+1-i} \big)^{d-i} \cdot \prod_{i=d-1}^{2(d-1)} P_{n+1}^{\langle m+1-d \rangle}   \]where $\Delta_\mathcal{P}$ is the discriminant of $\mathcal{P}(x)(p,q,r;a,b,c)$.
\end{theorem}

\begin{proof}
The proof is based on induction on $d$, Lemma \ref{lemma3} and Theorem \ref{theorem1}. It is essentially identical to the proof of Theorem 2.1 in the work of Tangboonduangjit and Thanatipanonda \cite{Aek2} and hence we skip it. 
\end{proof}

If we set $p=q=b=0$ and $r=a=c=1$, then we get the sequence of Fibonacci polynomials in $\mathcal{P}(x)(0,0,1;1,0,1)=(F_n(x))_{n\in\mathbb{Z}}$ where the sequence $(F_n(x))_{n\in\mathbb{Z}}$ is defined by \[F_0 (x)\equiv 0,\quad F_1 (x)\equiv 1, \quad F_{n+2}(x)=x F_{n+1}(x)+F_n (x).\]We recover the Fibonacci numbers and Pell numbers by evaluating $F_n(x)$ at $x=1$ and $x=2$ respectively. We note that $\Delta_\mathcal{P} =1$ and $\mathcal{U}(x)(0,0,1;1,0,1)=(F_n(x))_{n\in \mathbb{Z}}$. By Theorem \ref{theorem2}, Theorem \ref{theorem3}, Theorem \ref{theorem3.5}  and Theorem \ref{theorem4}, we get the following corollary:

\begin{corollary} \label{corollary2}
Let $m\geq 1$ and $d\geq 1$. Let $s,k, n$ be any integers. Let $(d_i)_{1\leq i\leq m}$ and $(e_i)_{1\leq i\leq m}$ be sequences of integers. Then
\begin{equation}
\begin{aligned}
 \det([(F_{s+k(n+i+j)}(x))^m]_{0\leq i,j\leq m})={}& (-1)^{(s+kn+1)\binom{m+1}{2}} \cdot\prod_{i=0}^m \binom{m}{i} (F_{(i+1)k}(x))^{2(m-i)},\\
 \det \Big( \Big[\prod_{f=j+1}^m F_{s+k(n+i+d_f)}(x) \prod_{g=1}^j F_{s+k(n+i+e_g)}(x)\Big]_{0\leq i,j\leq m}   \Big) ={} & (-1)^{(s+kn+1)\binom{m+1}{2}+k\binom{m+1}{3}} \prod_{l=1}^m (F_{kl}(x))^{m+1-l}\cdot\\ \nonumber  &\prod_{1\leq i\leq j\leq m} (-1)^{kd_j} F_{k(e_i-d_j)}(x),\\
\det\Big(\Big[\frac{1}{F_{s+k(n+i+j)}(x)}\Big]_{0\leq i,j\leq m}\Big)={}&\frac{(-1)^{(s+kn)\binom{m+1}{2}} \prod_{i=0}^m (F_{k(i+1)}(x))^{2(m-i)}   }{\prod_{0\leq i,j\leq m} F_{s+k(n+i+j)}(x)} ,\\
\det ([ (F_{n+i+j}(x))^{\langle m\rangle}]_{0\leq i,j\leq d-1} )={}&  (-1)^{n\binom{d}{2}+\binom{d+1}{3}}\cdot \prod_{i=1}^{d-1} \big( F_{i}(x) F_{m+1-i}(x) \big)^{d-i} \cdot \\ \nonumber & \prod_{i=d-1}^{2(d-1)} (F_{n+1}(x))^{\langle m+1-d \rangle}.
\end{aligned}
\end{equation}
\end{corollary}

\begin{remark}
We recover the identities (\ref{equation2}) and (\ref{equation4}) by setting $x=1$ in the first identity and the last identity in Corollary \ref{corollary2} respectively.
\end{remark}

\begin{remark}
We recover the results shown by Alfred \cite{Alfred2} by setting $x=1, s=0, k=1, n=0$ and $d_i\equiv 0$, $e_i\equiv 1$ for all $1\leq i,j\leq m$ in the second identity in Corollary \ref{corollary2}.
\end{remark}

\begin{remark}
We get the identity (\ref{new}) by setting $x=1$, $s=0$, $k=n=1$ in the third identity in Corollary \ref{corollary2}.
\end{remark}

If we set $p=2$, $q=a=c=1$ and $r=b=0$, then we get the sequence of Lucas polynomials in $\mathcal{P}(x)(2,1,0;1,0,1)=(L_n(x))_{n\in\mathbb{Z}}$ where the sequence $(L_n(x))_{n\in\mathbb{Z}}$ is defined by \[ L_0 (x)\equiv 2, \quad L_1(x)\equiv x, \quad L_{n+2}(x) =x L_{n+1}(x) +L_n(x).\]We recover the Lucas numbers by evaluating $L_n(x)$ at $x=1$. We note that $\Delta_\mathcal{P}=(-x^2 -4)$ and $\mathcal{U}(x)(0,0,1;1,0,1)=(F_n(x))_{n\in\mathbb{Z}}$. By Theorem \ref{theorem2}, Theorem \ref{theorem3}, Theorem \ref{theorem3.5}  and Theorem \ref{theorem4}, we get the following corollary:

\begin{corollary} \label{corollary3}
Let $m\geq 1$ and $d\geq 1$. Let $s,k, n$ be any integers. Let $(d_i)_{1\leq i\leq m}$ and $(e_i)_{1\leq i\leq m}$ be sequences of integers. 
\begin{equation}
\begin{aligned}
 \det([(L_{s+k(n+i+j)}(x))^m]_{0\leq i,j\leq m})={}& (-1)^{(s+kn)\binom{m+1}{2}} (x^2 +4)^{\binom{m+1}{2}} \cdot \\ \nonumber &\prod_{i=0}^m \binom{m}{i} (F_{(i+1)k}(x))^{2(m-i)},\\
 \det \Big( \Big[\prod_{f=j+1}^m L_{s+k(n+i+d_f)}(x) \prod_{g=1}^j L_{s+k(n+i+e_g)}(x)\Big]_{0\leq i,j\leq m}   \Big) ={} & (-1)^{(s+kn)\binom{m+1}{2}+k\binom{m+1}{3}} (x^2+4)^{\binom{m+1}{2}} \\ \nonumber &\prod_{l=1}^m (F_{kl}(x))^{m+1-l}\cdot\prod_{1\leq i\leq j\leq m} (-1)^{kd_j} F_{k(e_i-d_j)}(x),\\
\det\Big(\Big[\frac{1}{L_{s+k(n+i+j)}(x)}\Big]_{0\leq i,j\leq m}\Big)={}&\frac{(-1)^{(s+kn+1)\binom{m+1}{2}} (x^2+4)^{\binom{m+1}{2}}}{\prod_{0\leq i,j\leq m} L_{s+k(n+i+j)}(x)}\cdot \\ &\prod_{i=0}^m (F_{k(i+1)}(x))^{2(m-i)} \\
\det ([ (L_{n+i+j}(x))^{\langle m\rangle}]_{0\leq i,j\leq d-1} )={}& (-1)^{(n+1)\binom{d}{2}+\binom{d+1}{3}} (x^2+4)^{\binom{d}{2}}\cdot\\ \nonumber &\prod_{i=1}^{d-1} \big( F_{i}(x) F_{m+1-i}(x) \big)^{d-i} \cdot \prod_{i=d-1}^{2(d-1)} (L_{n+1}(x))^{\langle m+1-d \rangle}.
\end{aligned}
\end{equation}
\end{corollary}

If we set $p=q=1$, $a=2$, $c=-1$ and $r=b=0$, then we get the sequence of Chebyshev polynomials of the first kind in $\mathcal{P}(x)(1,1,0;2,0,-1)=(T_n(x))_{n\in\mathbb{Z}}$ where the sequence $(T_n(x))_{n\in\mathbb{Z}}$ is defined by \[T_0(x)\equiv 1, \quad T_1(x)\equiv x, \quad T_{n+2}(x)=2x T_{n+1}(x) -T_n(x).\]We note that $\Delta_\mathcal{P}=(-x^2+1)$.

If we set $p=1$, $a=q=2$, $c=-1$ and $r=b=0$, then we get the sequence of Chebyshev polynomials of the second kind in $\mathcal{P}(x)(1,2,0;2,0,-1)=(S_n(x))_{n\in\mathbb{Z}}$ where the sequence $(S_n(x))_{n\in\mathbb{Z}}$ is defined by \[S_0(x)\equiv 1, \quad S_1(x)\equiv 2x, \quad S_{n+2}(x)=2x S_{n+1}(x) -S_n(x).\]We note that $\Delta_\mathcal{P}=(-2x^2+1)$. 

We note that \[\mathcal{U}(x)(0,0,1;2,0,-1)=(U_n(x))_{n\in\mathbb{Z}}\]where \[U_n(x)=S_{n-1}(x)\text{ for }n\in\mathbb{Z} .\]We get two corollaries by Theorem \ref{theorem2}, Theorem \ref{theorem3}, Theorem \ref{theorem3.5} and Theorem \ref{theorem4}.

\begin{corollary} \label{corollary4}
Let $m\geq 1$ and $d\geq 1$. Let $s,k, n$ be any integers. Let $(d_i)_{1\leq i\leq m}$ and $(e_i)_{1\leq i\leq m}$ be sequences of integers. Then
\begin{equation}
\begin{aligned}\nonumber
\det([(T_{s+k(n+i+j)}(x))^m]_{0\leq i,j\leq m})={}& (x^2 -1)^{\binom{m+1}{2}} \cdot\\ &\prod_{i=0}^m \binom{m}{i} (S_{(i+1)k-1}(x))^{2(m-i)},\\
 \det \Big( \Big[\prod_{f=j+1}^m T_{s+k(n+i+d_f)}(x) \prod_{g=1}^j T_{s+k(n+i+e_g)}(x)\Big]_{0\leq i,j\leq m}   \Big) ={}& (x^2-1)^{\binom{m+1}{2}} \prod_{l=1}^m (S_{kl-1}(x))^{m+1-l}\cdot\\ &\prod_{1\leq i\leq j\leq m} S_{k(e_i-d_j)-1}(x),\\ 
\det\Big(\Big[\frac{1}{T_{s+k(n+i+j)}(x)}\Big]_{0\leq i,j\leq m}\Big)={}&\frac{ (-x^2+1)^{\binom{m+1}{2}}\prod_{i=0}^m (S_{k(i+1)-1}(x))^{2(m-i)}}{\prod_{0\leq i,j\leq m} T_{s+k(n+i+j)}(x)}, \\ 
 \det ([ (T_{n+i+j}(x))^{\langle m\rangle}]_{0\leq i,j\leq d-1} )={}& (-1)^{d\binom{d}{2}+\binom{d+1}{3}} (-x^2+1)^{\binom{d}{2}}\cdot \\ &\prod_{i=1}^{d-1} \big( S_{i-1}(x) S_{m-i}(x) \big)^{d-i} \cdot \prod_{i=d-1}^{2(d-1)} (T_{n+1}(x))^{\langle m+1-d \rangle}.   
\end{aligned}
\end{equation}
\end{corollary}

\begin{corollary} \label{corollary5}
Let $m\geq 1$ and $d\geq 1$. Let $s,k, n$ be any integers. Let $(d_i)_{1\leq i\leq m}$ and $(e_i)_{1\leq i\leq m}$ be sequences of integers. Then
\begin{equation}
\begin{aligned}
\det([(S_{s+k(n+i+j)}(x))^m]_{0\leq i,j\leq m})={}& (2x^2 -1)^{\binom{m+1}{2}} \cdot\prod_{i=0}^m \binom{m}{i} (S_{(i+1)k-1}(x))^{2(m-i)},\\
 \det \Big( \Big[\prod_{f=j+1}^m S_{s+k(n+i+d_f)}(x) \prod_{g=1}^j S_{s+k(n+i+e_g)}(x)\Big]_{0\leq i,j\leq m}   \Big)  ={} & (2x^2-1)^{\binom{m+1}{2}} \prod_{l=1}^m (S_{kl-1}(x))^{m+1-l}\cdot\\ \nonumber  & \prod_{1\leq i\leq j\leq m} S_{k(e_i-d_j)-1}(x),\\
\det\Big(\Big[\frac{1}{S_{s+k(n+i+j)}(x)}\Big]_{0\leq i,j\leq m}\Big)={}&\frac{ (-2x^2+1)^{\binom{m+1}{2}}\prod_{i=0}^m (S_{k(i+1)-1}(x))^{2(m-i)}}{\prod_{0\leq i,j\leq m} S_{s+k(n+i+j)}(x)}, \\ 
 \det ([ (S_{n+i+j}(x))^{\langle m\rangle}]_{0\leq i,j\leq d-1} )={}& (-1)^{d\binom{d}{2}+\binom{d+1}{3}} (-2x^2+1)^{\binom{d}{2}}\cdot\\  &\prod_{i=1}^{d-1} \big( S_{i-1}(x) S_{m-i}(x) \big)^{d-i} \cdot \prod_{i=d-1}^{2(d-1)} (S_{n+1}(x))^{\langle m+1-d \rangle}. 
\end{aligned}
\end{equation}
\end{corollary}

By Favard's theorem \cite{Favard} (see also the standard reference textbook by Chihara \cite[Chapter 2]{Chihara}), the sequence $\mathcal{P}(x)(1,q,r;1,b,c)=(P_n)_{n\in\mathbb{Z}_{\geq 0}}$ forms a sequence of orthogonal polynomials (with respect to certain linear functional) for $q\neq 0$ and $c\neq 0$. By Theorem \ref{theorem2}, Theorem \ref{theorem3}, Theorem \ref{theorem3.5} and Theorem \ref{theorem4}, we state some determinant identities for matrices containing (powers of) such orthogonal polynomials.

\begin{corollary}\label{corollary6}
Let $n\geq 0$, $m\geq 1$ and $d\geq 1$. Let $(P_n)_{n\in\mathbb{Z}_{\geq 0}}$ be a sequence of orthogonal polynomials of the form: \[P_0\equiv 1, \quad P_1\equiv qx+r,\quad P_{n+2}= (x+b) P_{n+1} +c P_n\]where $c\neq 0$, $q\neq 0$ and $r,b$ are any real numbers. Then 
\begin{equation}
\begin{aligned}\nonumber
\det([P_{n+i+j}^m]_{0\leq i,j\leq m})={}& (-1)^{(n+1)\binom{m+1}{2}}\cdot \Delta^{\binom{m+1}{2}}\cdot c^{n\binom{m+1}{2}+2\binom{m+1}{3}}\cdot \prod_{i=0}^m \Big(\binom{m}{i} U_{i+1}^{2(m-i)}   \Big),\\
 \det \Big( \Big[\prod_{f=j+1}^m P_{n+i+d_f} \prod_{g=1}^j P_{n+i+e_g}\Big]_{0\leq i,j\leq m}   \Big) = {}&(-\Delta)^{\binom{m+1}{2}} (-c)^{n\binom{m+1}{2} +\binom{m+1}{3}} \prod_{l=1}^m U_{l}^{m+1-l}\cdot\prod_{1\leq i\leq j\leq m} (-c)^{d_j} U_{e_i-d_j},\\
\det\Big(\Big[\frac{1}{P_{s+k(n+i+j)}}\Big]_{0\leq i,j\leq m}\Big)={}&\frac{(-1)^{(s+kn)\binom{m+1}{2}} c^{(s+kn)\binom{m+1}{2}+2k\binom{m+1}{3}}\Delta^{\binom{m+1}{2}} \prod_{i=0}^m U_{k(i+1)}^{2(m-i)}   }{\prod_{0\leq i,j\leq m} P_{s+k(n+i+j)}},\\
 \det ([ P_{n+i+j}^{\langle m\rangle}]_{0\leq i,j\leq d-1} )={}&(-1)^{n\binom{d}{2}+\binom{d+1}{3}} c^{(n+d-2)\binom{d}{2}} \Delta^{\binom{d}{2}} \cdot \prod_{i=1}^{d-1} \big( U_{i} U_{m+1-i} \big)^{d-i} \cdot \prod_{i=d-1}^{2(d-1)} P_{n+1}^{\langle m+1-d \rangle} .
\end{aligned}
\end{equation}where $\Delta=(q^2 -q)x^2+(qr-r-bq)x+(r^2-br-c)$ and $(U_n)_{n\in\mathbb{Z}_{\geq 0}}$ is the sequence of orthogonal polynomials defined by \[U_0\equiv 0,\quad U_1\equiv 1,\quad U_{n+2} = (x+b) U_{n+1} + c U_n.\]
\end{corollary}

\end{document}